\documentclass[11pt,a4paper]{amsart}
\usepackage{amssymb,amsfonts,amsmath}
\usepackage{slashbox,booktabs}
\usepackage{amsthm}
\usepackage{empheq}
\usepackage[english]{babel}
\usepackage[utf8]{inputenc}
\usepackage{latexsym}
\usepackage{epsfig}
\usepackage[all]{xy}
\usepackage{tikz}
\usepackage{cancel}

\newcommand*\widefbox[1]{\fbox{\hspace{2em}#1\hspace{2em}}}

\usepackage{amsmath, amscd}

\makeatletter
\renewenvironment{proof}[1][\proofname]{%
	\par\pushQED{\qed}\normalfont%
	\topsep6\p@\@plus6\p@\relax
	\trivlist\item[\hskip\labelsep\bfseries#1\@addpunct{.}]%
	\ignorespaces
}{%
	\popQED\endtrivlist\@endpefalse
}
\makeatother

\newtheorem{theorem}{Theorem} [subsection]
\newtheorem{lemma}[theorem]{Lemma}
\newtheorem{proposition}[theorem]{Proposition}
\newtheorem{corollary}[theorem]{Corollary}

\newtheorem*{question*}{Question}

\newcommand{\CC}{{\mathbb{C}}}
\newcommand{\QQ}{{\mathbb{Q}}}

\newcommand{\PP}{{\mathbb{P}}}

\newcommand{\m}{\mathcal}
\newcommand{\Gr}{{\textbf{Gr}(2,4)}}

\begin{document}

\title{Counting surfaces singular along a line in $\PP^3$}
\author{Shachar Carmeli, Lev Radzivilovsky } 

\date{\today}
\maketitle 

\begin{abstract}
	We enumerate the surfaces of degree $d$ in $\PP^3$ having a singular
	line of order $k$, passing through $\delta$ generic points (where $\delta$ is 
	the dimension of moduli space of such surfaces). 
\end{abstract}

\section{Introduction}
Consider surfaces in $\PP^3$ of degree $d$, having a singularity of degree $k$ along 
a line (which is not fixed).
Such surfaces may be described by homogeneous polynomials of degree $d$,
for which there exists a line such that the polynomial of the surface is in 
the $k$-th power of the ideal of the line. 
The moduli space of such surfaces is of dimension 
$\delta= {d - k + 2 \choose 2}\cdot \frac{{d + 2k + 3}}{3} + 3$.
The natural enumerative question regarding such surfaces is:

\begin{question*}
	\label{qsn: main question}
	Given $\delta$ generic points in $\PP^3$, how many surfaces of degree $d$ singular along a line to degree $k$ pass through those $\delta$ points? 
\end{question*}
In other words, what is the degree of the locus of such surfaces in the projective space of all surfaces of degree $d$?

Denote by $N$ the number of such surfaces.
The answer to this question is given by the following
formula
\[N=\frac{\phi}{4} \left[  (3 \phi)^3 -2(9 \phi^2+1)(d-2k+1) + 
\phi \left(2d(d+1)-12k(d-k+1)+5\right)  \right],\]
where $\phi = \frac{1}{12}(d-k+2)(d-k+1)(k+1)k$.

Our interest in this question started with surfaces of degree $4$ having 
a singular line, but the simplest case is surfaces of degree $2$ 
having a singular line, which are just unions of two planes. To specify $2$ planes, 
we must fix $6$ points, since in order to specify one plane, we need $3$ points. Given $6$ generic points, there are $10$ ways to split them into $2$ triples, and each $3$ generic points specify a unique plane, so in that simplest particular case, $N=10$.

Similar problems have been studied by several authors. For example, the fact that the formula for $N$ is a polynomial in $d$  follows a-priori from the results in \cite{CLV}. Enumerative properties of varieties singular along a sub-variety of positive dimension has also been studied in \cite{KKN} from a different perspective. A similar problem with a point instead of a line where studied in \cite{K}. 
The special case $d=3$ and $k=2$ of our calculation can be obtain using different methods, as carried in \cite{CV,C,M}.  
\subsection{Acknowledgments}

The authors want to thank Eugenii Shustin for introducing us to Enumerative Geometry.
We also want to thank Liam Hanany for fixing mistakes in the previous version of 
the computation (which covered only a particular case but was much longer,
so verifying it was much more challenging), and Guy Kapon for his interest. 
Finally, we are very grateful to the anonymous referee 
for referring us to the enumerative computations for the case of cubic scrolls and for many other critical remarks, which allowed us to correct a mistake in our calculations and improve the presentation.   

S. Carmeli is partially supported by the Adams fellowship of the Israel Academy of Science and Humanities. 
\section{Tools and Notations}
We introduce tools and notations which are necessary for our computations. Large parts of this section are well-known, and we recall them here only to make the text more readable and self-contained.
\subsection{The Space of Lines in Projective Space: $\Gr$}
The space of lines in $\PP^3$, which is the same as the space of two-dimensional linear subspaces in $\CC^4$, is denoted $\Gr$. The standard textbook about Grassmanians is [GH] 
(the last part of chapter 1). 
Since we are going to do computations related to the cohomologies of $\Gr$, 
we remind the multiplication table of the cohomology ring of $\Gr$. 
The cohomology classes in different dimensions are 

\[\begin{array}{*{20}{c}}
0& & &2& &4& &6& & &8\\
1,& & &{\sigma_1},& &\sigma_{1,1}, \sigma_2,& &\sigma_{2,1},& & &\sigma _{2,2}.
\end{array}\]

Geometrically, $\sigma_1$ is represented by the collection of lines intersecting a specific line,
$\sigma_{1,1}$ by the lines inside a specific plane, $\sigma_{2}$ by the lines passing through a point,
$\sigma_{2,1}$ by the lines simultaneously being in a specific plane and passing through a specific point in that
plane, and $\sigma_{2,2}$ by a specific line. 

We shall use the multiplication table for $H^*(\Gr)$:
\[ \begin{array}{*{90}{c}}
\sigma_1   \cdot \sigma_1       = \sigma_{1,1} + \sigma_2& &
\sigma_1   \cdot \sigma_{1,1}   = \sigma_{2,1}& &
\sigma_1   \cdot \sigma_2       = \sigma_{2,1}& &
\sigma_1   \cdot \sigma_{2,1}   = \sigma_{2,2}\\
\sigma_{1,1} \cdot \sigma_{1,1} = \sigma_{2,2}& & 
\sigma_2     \cdot \sigma_2     = \sigma_{2,2}& & 
\sigma_{1,1} \cdot \sigma_2     = 0
\end{array}\]
The moduli space of surfaces having a singularity of degree $k$ along a line 
will be denoted $SingL_{k,d}$.

Let $V_{k,d}$ be the vector bundle on $\Gr$ with fiber at a 2-dimensional subspace $\ell$ being the vector space of all polynomials of degree $d$ vanishing on $\ell$ to the $k$-th order. 
Let $\PP V_{k,d}$ be the projectivization of that bundle, which is the bundle of projective spaces with fiber at $\ell$ given by $\PP (V_{k,d}|_\ell)$.
There is a natural forgetful map $forg:\PP V_{k,d} \to SingL_{k,d}$ given by $forg(\ell, f) = f$.

Now we compute the dimension $\delta$ of $SingL_{k,d}$. The map $forg$ is a birational map and hence $\delta=\dim(\PP V_{k,d})$. 
The dimension of $\Gr$ is 4, the dimension of the fiber $V_{k,d}|_\ell$ for $\{x=y=0\}$ is equal to the number of monomials $x^Ay^Bz^Ct^D$, such that $A+B+C+D=d$ and $A+B\geq k$. Projectivization
reduces the dimension by $1$.
So the computation of $\delta$ boils down to a combinatorial of separating $k$ stars written down in a line by $3$ bars, so that there are at least $k$ stars before the second bar.
If there are at least $k$ stars even before the first bar, the number of possibilities 
is ${d-k+3 \choose 3}$. Otherwise, there are $k$ possible placements for the first bar, 
and it remains to place $2$ bars among $d-k$ stars, which gives $k\cdot{d-k+2 \choose 2}.$
Combining the above we get
\[\delta={d - k + 2 \choose 2}\cdot \frac{{d + 2k + 3}}{3} + 3.\]

A requirement for a surface to pass via a point in $\PP^3$ specifies a divisor,
which corresponds to $h \in H^2(SingL_{k,d})$ by Poincare duality. 
Our question is about the intersection number on the moduli space,
and it is equivalent to the computation of $h^\delta$. 
We shall perform this computation in $H^*(\PP V_{k,d})$, which is equivalent.

The standard formula for the cohomology ring of the projectivization  of a vector bundle in terms of the Chern polynomial is (for instance \cite{H}, Appendix A) is: 
\[H^*(\PP \mathcal{E}) = H^*(\mathcal{B})[t]/c_\mathcal{E}(t),\] 
where $\mathcal{B}$
is the base space, $t$ is the cohomology class of the hyperplane section, 
\[c_\mathcal{E}(t) = t^n+c_1t^{n-1}+c_2t^{n-2}+...,\] a. k. a. the Chern polynomial, and $n$ is the rank of the bundle.

Firstly, note that $forg^*(h)=t$. Indeed, $t$ is defined by a linear equation, 
which outside of a singular line (which is of codimension $2$) specifies 
a hyperplane section in the fiber.
Therefore, it suffices to compute $t^\delta$ modulo the Chern polynomial of $t$, 
and extract the coefficient of $t^{\delta-4}\cdot \sigma_{2,2}$.

We shall see that it is even better to compute the inverse Chern polynomial 
rather than the Chern polynomial. 
The inverse Chern polynomial is defined as follows:
given Chern polynomial \[c=1+c_1+c_2+...,\] where $c_i\in H^{2i}$,
the inverse Chern polynomial is $c^{-1}=1+d_1+d_2+...$, 
where $d_i\in H^{2i}$, so that $c^{-1}\cdot c= 1$.

We may write the Chern polynomial with a formal variable $t$ 
(in such a way that if $t$ would be in $H^2$, the Chern polynomial 
would be homogeneous) 
to keep track of cohomological degrees. When we substitute $t=1$, 
we get the usual Chern polynomial. For instance, we may write:
\[c(t)=t^{\delta-3}+t^{\delta-4}c_1+t^{\delta-5}c_2+t^{\delta-6}c_3+t^{\delta-7}c_4,\]
\[c^{-1}(t)=t^3+t^2d_1+t\cdot d_2+d_3+t^{-1}d_4.\]
The product of these two polynomials is $t^\delta$.

Hence, if we substitute $t=forg^*(h)$, and $c_\mathcal{E}(t) = 0$,
then the product of $c(t)$ and any other cohomology class is zero, so 
\[c(t) \cdot (t^3+t^2d_1+t\cdot d_2+d_3) = 0.\]
The same reasoning does not apply to $c(t) \cdot t^{-1}$, since it is not an actual cohomology class. Hence, we get the formula 
\[t^{\delta}=t^{\delta-4}d_4+t^{\delta-5}c_1d_4+... \pmod{c(t)}.\]
Therefore, if we wanted to compute $t^\delta$ modulo the Chern polynomial of $t$, 
and extract the coefficient of $t^{\delta-4}\cdot \sigma_{2,2}$, the answer is $d_4$.

Hence, what we really need is the \textbf{free coefficient of the inverse Chern polynomial}.

\subsection{Resolution of the Bundle}

Let $X=\Gr\times \PP^3$ and let $Fl_{1,2,4}=Z\subseteq X$ denote the subvariety of pairs of a line $\ell$ and plane $U$ such that $\ell\subseteq U$ (flags of type $(1,2,4)$). 

Let $I_Z$ denote the ideal sheaf of $Z$ in $X$. 
Let $\pi_1: X\to \Gr$ and $\pi_2:X\to \PP^3$ denote the projections.
If $F$ is a sheaf on $X$ we denote by $F(d)$ the sheaf $F\otimes \pi_2^*\m{O}_{\PP^3}(d)$.

Recall that $V_{k,d}$ denote the vector bundle on $\Gr$ which assigns with each line 
$\ell$ the vector space of homogeneous polynomials on $V=\CC^4$ that vanishes to order $k$ on $\ell$.
Denote by $P(d)$ the vector space of homogeneous polynomials of degree $d$ in $4$ variables, so that 
\[P(d)\cong \Gamma(\PP^3,\m{O}_{\PP^3}(d)).\] 
We have a natural embedding $\phi_{k,d}: V_{k,d} \to \mathcal{O}_{\Gr}\otimes_\CC P(d)$. 
The sections of $\mathcal{O}_{\Gr}\otimes_\CC P(d)$ over an open set $U$ have the form 
$\sum_J f_J \otimes x^J$ where for every multi-index $J=(j_1,j_2,j_3,j_4)$ we have $x^J:= x_1^{j_1}x_2^{j_2}x_3^{j_3}x_4^{j_4}$
and the $f_J$-s are some regular functions on $U$.
Such a section belongs to $V_{k,d}$ exactly when for every line $\ell \in U$, the homogeneous polynomial 
$\sum_J f_J(\ell)x^J$ vanishes to order $k$ on $\ell$, or in other words if 
$\sum_J f_J(\ell)x^J\in \Gamma(\PP^3,I_\ell^{k}(d))$ for every $\ell \in U$. 

On the other hand, we have a natural embedding $(\pi_1)_*(I_Z^k(d))\subseteq (\pi_1)_* \m{O}_X(d)\cong \mathcal{O}_{\Gr}\otimes_\CC P(d)$ so we can consider it as a sub-sheaf of $\mathcal{O}_{\Gr}\otimes_\CC P(d)$.  
The sections of this sub-sheaf are those sections $\sum_J f_J x^J$ such that the section 
$\sum_J (f_J \circ \pi_1) x^J$ of $\m{O}_X(d)$ vanish on $Z$ to order $k$.

\begin{proposition}
	\label{prop: equality of important bundle to push of ideal}
	The two sub-sheaves $V_{k,d}$ and $(\pi_2)_*I_Z^k(d)$ of $\m{O}_{\Gr}\otimes_\CC P(d)$ are the same.
\end{proposition}  

This result follows from a general lemma that we shall now state: 

\begin{lemma}
	Let $f:X\to Y$ be a smooth morphism of smooth algebraic varieties. Let $Z\subseteq X$ be a smooth subvariety of $X$ such that the restriction of $f$ to $Z$ remains smooth. Let $L$ be a line bundle on $X$. 
	Then, a section $\alpha$ of $L$ on an open set $V$ of $X$ is in $(I_Z^k \otimes L)(V)$ if and only if for every 
	$y \in Y$, the restriction to the fiber $\alpha|_{X_y\cap V}$ comes from a section of $I_{Z\cap X_y}^k\otimes L|_{X_y}$.   
\end{lemma}

\begin{proof}
	Clearly if $\alpha \in (I_Z^k \otimes L)(V)$, then its restriction to each fiber $X_y$ vanishes to order $k$ on $Z\cap X_y$. We prove the converse by induction on $k$. The case $k=1$ follows from the fact that a section of a line bundle vanishes on a subvariety if and only if it vanishes on each point of this subvariety, together with the fact that the preimages of points of $Y$ cover $X$. If $\alpha|_{X_y}$ vanishes to order $k$ on $Z\cap V \cap X_y$, then by the inductive hypothesis $\alpha \in I_Z^{k-1}\otimes L(V)$. It is enough then to show that the reduction $\bar{\alpha}$ of $\alpha$ to the quotient 
	$(I_Z^{k-1}/I_Z^k \otimes L)(V)\cong Sym^{k-1}(N^*_{Z/X})\otimes L(V)$ vanishes, where $N^*_{Z/X}$ is the conormal bundle. By the assumptions that $f$ is smooth along with its restriction to $Z$, we get that the natural restriction map gives an isomorphism $N^*_{Z/X}|_{X_y}\cong N^*_{Z\cap X_y/X_y}$ and hence an isomorphism 
	
	\[Sym^{k-1}(N^*_{Z/X}) \otimes L|_{X_y}\cong Sym^{k-1}(N^*_{Z\cap X_y/X_y})\otimes L|_{X_y}.\] 
	
	By the assumption on $\alpha$, we know that the reduction of $\alpha|_{X_y}$ to $Sym^{k-1}(N^*_{Z\cap X_y/X_y})\otimes L|_{X_y}$ vanishes, and the isomorphism above gives us that $\bar{\alpha}$ vanishes at every point of $Z$, so since it is a section of a vector bundle on $Z$, it vanishes. 
\end{proof}

Proposition \ref{prop: equality of important bundle to push of ideal} follows from this lemma by choosing $X =\Gr\times \PP^3$, 
$L=\m{O}_X(d)$, $f=\pi_2$, $Y=\Gr$,  $Z=Fl_{1,2,4}$ and $V$ to be any open set of the form $\pi_2^{-1}(U)$. Then the proposition exactly tells us that the sections of $V_{k,d}$ on $U$ are the same as the sections of $(\pi_2)_*I_Z^k(d)$, 
assuming that $X,Y,Z, f, f|_Z$ are all smooth, which is clear in this case.

We shall now construct a locally free resolution of $I_Z^k$ over $X$. 
Then, after sufficient twist, its push-forward down to $\Gr$ will provide a resolution of $V_{k,d}$ using standard sheaves with known Chern classes, allowing computation of the inverse Chern polynomial of $V_{k,d}$.  

Let $\tau$ be the tautological bundle on $\Gr$, and
$\nu = (\m{O}_{\Gr}^4/\tau)^*$ denote the dual of the tautological quotient vector bundle on $\Gr$, that on a line $\ell$ gives the linear functionals on $\CC^4$ vanishing on $\ell$. 
We have a canonical map given by multiplication $\pi_2^*\nu(-1)\to \m{O}_{X}$ with image $I_Z$. 
Recall that for an algebraic variety $M$, a vector bundle $E$ on $M$ and a functional 
$\phi:E\to \m{O}_M$ with zero locus $Z(\phi)$ there is an associated Kozul complex of the form 
\[0\to \wedge^n E \to \wedge^{n-1}E \to ...\to E\stackrel{\phi}{\to} \mathcal{O}_M\to \mathcal{O}_{Z(\phi)} \to 0.\] 
This complex is exact if $\phi$ is transversal to the zero section of $E^*$ (e.g. by \cite[Proposition 1.4 (b)]{GKZ}). 
We get a complex of the form 
\begin{equation}
\label{eqn: Koszul complex}
0\to \wedge^2\pi_2^*\nu(-1)\to \pi_2^*\nu(-1) \to I_Z \to 0. 
\end{equation}

To show that it is exact, it suffices to show that the functional 
$\pi_2^*\nu(-1)\to \mathcal{O}_X$ is transversal to the zero section. This is an easy computation in coordinates. 

In order to use the resolution (\ref{eqn: Koszul complex}) of $I_Z$ to resolve $I_Z^k$, we use some computations of symmetric powers. 

\subsubsection{Derived Symmetric Powers of Ideals} 

We first introduce the derived symmetric power functor. This functor will be used to resolve the powers of $I_Z$. 

Let $M$ be an algebraic variety over a field of characteristic $0$ and let $F\in Coh(M)$ be a coherent sheaf. The group $S_k$ acts on the $k$-fold derived tensor product 
\[F\otimes^LF\otimes^L...\otimes^LF:= F^{\otimes^L k}.\] 
We denote by $LSym^k(F)$ the derived fixed points of $S_k$ in this complex. 
In explicit terms, one chooses a locally-free resolution $E_\bullet \to F$ and then consider the actual $S_k$-invariants in the complex $(E_\bullet)^{\otimes k}$.

The main result we use regarding symmetric powers is the following: 

\begin{theorem}
	\label{thm: powers of ideal of codimension 2}
	Let $M$ be a smooth algebraic variety and let $Z$ be a smooth subvariety of codimension at most $2$. Then the multiplication map 
	$LSym^k(I_Z)\to I_Z^k$ is a quasi-isomorphism.
\end{theorem}    

Explicitly, this means that if $E_\bullet \to I_Z$ is a locally-free resolution, then 
the complex $(E_\bullet^{\otimes k})^{S_k}$ has only $0$-th homology and the $0$-th homology is $I_Z^k$. 

\begin{proof}
	The case where $Z$ is of codimension $1$ is much easier than the codimension 2 case and is left to the reader. From now on, we assume that $Z$ is of codimension $2$.  
	Since we ask if a map of complexes of sheaves is a quasi-isomorphism,  the statement is local on $M$ so we may assume that $M$ is the spectrum of a Noetherian regular local ring $R$ with regular sequence 
	$x_1,...,x_n$, and that $I_Z=(x_3,...,x_n)$. Let $V=span(e_1,e_2)$ be a two dimensional vector space and consider the Koszul resolution 
	\[0 \to \wedge^2 V\otimes R \stackrel{d}{\to} V\otimes R \stackrel{d}{\to} I_Z\to 0\] 
	with $d(e_i)=x_i$, which is exact because $Z$ is a complete intersection in $spec(R)$. It follows from a straightforward computation of the symmetric power of this resolution that $LSym^k(I_Z)$ is representable by the complex 
	\[
	\big[ \wedge^2 V\otimes Sym^{k-1}(V)\otimes R \stackrel{d}{\to}  Sym^k(V) \otimes R\big], 
	\] 
	where the differential is given by $d(a\wedge b \otimes \alpha \otimes r) = a\alpha \otimes d(b)r - b\alpha \otimes d(a)r$. 
	It remains to verify that the sequence 
	\[0\to \wedge^2 V\otimes Sym^{k-1}(V)\otimes R \stackrel{d}{\to}  Sym^k(V) \otimes R \stackrel{d}{\to} I_Z^k\to 0\] 
	is exact. 
	
	The map $Sym^k(V)\otimes R\to I_Z^k$ is onto since every element of $I_Z^k$ is a sum of products of $k$ elements of $I_Z$. 
	Consider the kernel of this map. 
	Let $\alpha = \sum_{j=0}^\ell e_1^je_2^{k-j}\otimes f_j$ be an element of the kernel. 
	We shall prove by induction on $\ell$ that it is a boundary. The case $\ell = 0$ is clear.
	We have 
	\[0=d(\alpha)=\sum_{j=0}^\ell x_1^jx_2^{k-j} f_j=
	x_2^{k-\ell}\sum_{j=0}^\ell x_1^j x_2^{\ell-j} f_j.\]  
	
	Since the sequence $x_1,...,x_n$ is regular, $x_2$ is not a zero divisor so $\sum_{j=0}^\ell x_1^j x_2^{\ell-j} f_j=0$. 
	Taken modulo $x_2$ this gives 
	$x_1^\ell f_\ell \equiv 0 \mod x_2$. Since $x_1$ is not a zero divisor modulo $x_1$, this means that $f_\ell = x_2 f_\ell'$ for some $f_\ell'$. Subtracting \[d(e_1\wedge e_2 \otimes e_1^{\ell-1} e_2^{k-\ell} \otimes f_j')= e_1^\ell e_2^{k-\ell} x_2f_j'-e_1^{\ell -1}e_2^{k-\ell+1} x_1f_j'\] 
	we see that modulo boundaries $\alpha$ is equivalent to an element of the same form with $f_\ell=0$. By induction, $\alpha$ is a boundary. 
	
	It remains to show that 
	$d:\wedge^2 V\otimes Sym^{k-1}(V)\otimes R \stackrel{d}{\to}  Sym^k(V) \otimes R$ is injective. Considered modulo $x_1$, this map is given by 
	multiplication by $e_1\otimes x_2$. Since this element is not a zero-divisor of the ring $Sym(V)\otimes R / x_1$, every element in the kernel of $d$ is divisible by $x_1$. 
	Then, using again that $x_1$ is not a zero divisor, we get that the elements of $Ker(d)$ are infinitely divisible by $x_1$, hence equal to $0$.    
\end{proof}

\subsubsection{Resolving the Powers of the Ideal Sheaf of the Incidence Variety}  

We shall now combine the Koszul resolution of $I_Z$ and Theorem \ref{thm: powers of ideal of codimension 2} to produce an Eagon-Northcott type resolution to $I_Z^k$ (see \cite[Appendix B]{L}). 

The subvariety $Z\subseteq X$ is a smooth subvariety of codimension $2$, so Theorem \ref{thm: powers of ideal of codimension 2} applies to $Z$. It follows that: 

\begin{proposition}
	The canonical map $LSym^k(I_Z)\to I_Z^k$ is a quasi-isomorphism. 
\end{proposition}

Since we have a locally-free resolution 
\[
0\to \wedge^2\pi_2^*(\nu(-1))\to \pi_2^*\nu(-1) \to I_Z\to 0,
\]
we get a resolution to the powers of $I_Z$ by taking the symmetric powers of this locally free resolution. We deduce the following result:

\begin{proposition}
	\label{proposition: the sequence}
	There is a short exact sequence on $X$ of the form 
	\[0\to \wedge^2(\pi_2^*\nu(-1))\otimes Sym^{k-1}(\pi_2^*\nu(-1))\stackrel{mul_1}{\to} Sym^k(\pi_2^*\nu(-1))\stackrel{mul_0}{\to} I_Z^k\to 0.\] 
\end{proposition}

To get from the sequence in Proposition \ref{proposition: the sequence} a resolution of $V_{k,d}$, twist by $d$ and apply the functor $R(\pi_2)_*$. 
We have a quasi-isomorphism 
\[
\big[\wedge^2(\pi_2^*\nu(-1))\otimes Sym^{k-1}(\pi_2^*\nu(-1))(d)\stackrel{mul_1}{\to} Sym^k(\pi_2^*\nu(-1))(d)\big]\cong I_Z^k(d) 
\]
in the derived category $D^b(X)$ of bounded complexes of coherent sheaves on $X$. We can rewrite the complex 
$[\wedge^2(\pi_2^*\nu(-1))\otimes Sym^{k-1}(\pi_2^*\nu(-1))(d)\stackrel{mul_1}{\to} 
Sym^k(\pi_2^*\nu(-1))(d)]$ as 
\[
	[\wedge^2(\pi_2^*\nu)\otimes Sym^{k-1}(\pi_2^*\nu)(d-k-1)\stackrel{mul_1}{\to} Sym^k(\pi_2^*\nu)(d-k)].
\]

By the projection formula and base-change, we have 
$R(\pi_2)_*(\pi_2^*F(m))\cong F\otimes_\CC R\Gamma(\PP^3,\mathcal{O}_{\PP^3}(m))$ for every coherent sheaf $F$ on $\Gr$. 
In particular, for every coherent sheaf $F$ on $\Gr$ and every $m\ge -3$, the sheaf $\pi_2^*F(m)$ is acyclic to the functor  $R(\pi_2)_*$, so for $d-k-1 \ge -3$ we get 
\begin{align*}
&R(\pi_2)_*[\wedge^2(\pi_2^*\nu)\otimes Sym^{k-1}(\pi_2^*\nu)(d-k-1)\stackrel{mul_1}{\to} Sym^k(\pi_2^*\nu)(d-k)]\cong \\    
&[\wedge^2 \nu \otimes Sym^{k-1}\nu \otimes_\CC P(d-k-1) \to Sym^k(\nu) \otimes_\CC P(d-k)].  
\end{align*}
As a result, we deduce: 
\begin{proposition}
	\label{proposition: the resolution}
	For every $d\ge k-3$, the higher cohomologies of $R(\pi_2)_* I_Z^k(d)$ vanish. Moreover, for those values of $d$ we have a short exact sequence 
	\[0\to \wedge^2 \nu \otimes Sym^{k-1}\nu \otimes_\CC P(d-k-1) \to Sym^k(\nu) \otimes_\CC P(d-k)\to 
	V_{k,d} \to 0.
	\] 
\end{proposition}

\begin{corollary}
	\label{K-class}
	The equation
	\begin{equation} 
	\boxed{[V_{k,d}]= {d-k+3\choose 3}[Sym^k(\nu)] - {d-k+2 \choose 3}[\wedge^2 \nu][Sym^{k-1}\nu]}
	\end{equation}

	holds in $K^0(\Gr)$. 
\end{corollary}

\begin{proof}
	Just note that $dim(P(m))={m+3 \choose 3}$. 
\end{proof}






\section{Chern Characters} 
Recall that, for a vector bundle $\mathcal{E}$ over an algebraic variety $X$, the Chern character is defined by the formula $ch(\mathcal{E})=\sum_i e^{x_i}$ where $x_i$ are virtual classes, having the property that 
\[c(\mathcal{E})=(1+x_1)(1+x_2)...(1+x_n).\]
In other words, $\sigma_i(x_1,...,x_n)=c_i(\mathcal{E})$. Chern characters are extremely useful for our computations, since $ch:K^0(X)\to H^{even}(X,\mathbb{Q})$ is a ring homomorphism.

Our final aim is to compute the inverse Chern polynomial of $V_{k,d}$. It 
is much easier first to compute the Chern character and then to convert it to 
the Chern polynomial. 
Right now, we want to recall $ch(Sym^k(\nu))$, which will be the building block
for the further computations.

\subsection{Chern Classes and Characters of $\nu$}
We start with $c(\tau^*)$. Since $\tau^*$ is two-dimensional, 
we need to compute just $c_1$ and $c_2$.
The first Chern class is given by the locus of linear dependence of two generic sections. 
In our case, we can choose the two functionals $x$ and $y$ and then the collection of lines 
on which these two are linearly dependent is the collection of lines $\ell$ 
for which $ax+by$ vanishes on the corresponding linear 2-space, and this
is equivalent to $\ell$ intersecting a specific line ($\{x=y=0\}$). 
This is the cycle representing $\sigma_1$.
The second Chern class is represented by zero set of a generic section.
Take the section $x$. Its zeroes are the lines which belong to a fixed plane,
hence representing $\sigma_{1,1}$. To summarize:
\[c(\tau^*)=1 + \sigma_1 + \sigma_{1,1}.\]
In $K^0(\Gr)$, we have $[\nu]+[\tau^*]=[V^*]=4[\CC]$ and hence 
\[ c(\nu) = \frac{1}{c(\tau^*)}= 1 - \sigma_1 + \sigma_{2}.\]

Now we want to convert to Chern characters.
For this we introduce the virtual cohomology classes $x,y$, 
such that 
\[ c(\nu) = 1 - \sigma_1 + \sigma_{2}=(1+x)(1+y).\]
In other words, $x+y=-\sigma_1$ and $xy=\sigma_2$.

\begin{align*}
& ch(\nu) = e^x + e^y=\\
&= 2 + (x + y) + \frac{(x+y)^2-2xy}{2} + \frac{(x+y)((x+y)^2-3xy)}{6}    
+ \frac{((x+y)^2-2xy)^2 - 2x^2y^2}{24} =\\
&=2 -\sigma_1 +\frac{(-\sigma_1)^2 -2\sigma_2}{2} + \frac{(-\sigma_1)((-\sigma_1)^2-3\sigma_2)}{6}          + \frac{((-\sigma_1)^2-2\sigma_2)^2 - 2\sigma_2^2}{24}  =\\
&=2 -\sigma_1 +\frac{\sigma_{1,1}-\sigma_2}{2} + \frac{\sigma_{2,1}}{6}.
\end{align*}

We conclude
\begin{equation}
\label{ch_nu}
ch(\nu) = 2 -\sigma_1 +\frac{\sigma_{1,1}-\sigma_2}{2} + \frac{\sigma_{2,1}}{6}.
\end{equation}

We shall also need $ch(\wedge^2 \nu)$. Notice that $\wedge^2 \nu$ is a line bundle and
it's first Chern class is $x+y=-\sigma_1$, so 
\begin{equation}
\label{ch_wedge_2}
ch(\wedge^2 \nu)=e^{x+y}=e^{-\sigma_1} = 1 -\sigma_1 + \frac{\sigma_{1,1}+ \sigma_2}{2} - \frac{\sigma_{2,1}}{3} + \frac{\sigma_{2,2}}{12}.
\end{equation}

We shall now compute the Chern character of $V_{k,d}$. In fact, we present two approaches to this computation. One is more straightforward but slightly technically harder, and one which uses the Adams operations. The two approaches give the same answer, hence providing additional evidence for the validity of the computation. We expect that the second approach can be generalized to other situations of this kind.

\subsection{Chern Characters of $ch(Sym^t(\nu))$ - First Computation.}

We turn now to compute $ch(Sym^t(\nu))$. We have
\[ch(Sym^t(\nu))=\sum_j e^{jx}e^{(t-j)y}=\frac{e^{(t+1)x}-e^{(t+1)y}}{e^x-e^y}.\]
Let $a=e^x-1$ and $b=e^y-1$. Then we get 

\begin{align}
\label{ch_sym} &ch(Sym^t(\nu))=\frac{(a+1)^{t+1}-(b+1)^{t+1}}{a-b} = \\
\nonumber &=(t+1) + {t+1 \choose 2} (b+a) + 
{t+1 \choose 3} \left((a+b)^2 - ab\right) + \\
\nonumber &+{t+1 \choose 4} (a+b)\left((a+b)^2 - 2ab\right) + 
{t+1 \choose 5}\left(a^2b^2 + (a+b)^2\left((a+b)^2 - 3ab\right)\right). 
\end{align}

We stop the expansion at degree 4 since $\Gr$ is 4-dimensional, 
hence cohomologies above $H^8$ are zero. 

Note that by formula (\ref{ch_nu}) we have 
\[a+b=ch(\nu)-2=-\sigma_1 +\frac{\sigma_{1,1}-\sigma_2}{2} + \frac{\sigma_{2,1}}{6}.\]
Also, 
\begin{align*}
ab&= (e^x-1)(e^y-1) = e^{x+y}-e^x-e^y + 1 = \\
&= 1 -\sigma_1 + \frac{\sigma_{1,1}+ \sigma_2}{2} - \frac{\sigma_{2,1}}{3} + \frac{\sigma_{2,2}}{12} - 2 +\sigma_1 -\frac{\sigma_{1,1}-\sigma_2}{2} - \frac{\sigma_{2,1}}{6} + 1.
\end{align*}     
In other words,
\[
ab=\sigma_2 -\frac{\sigma_{2,1}}{2} + \frac{\sigma_{2,2}}{12}.
\]
We shall also need the formulas
\[
(a+b)^2 =\left(-\sigma_1 +\frac{\sigma_{1,1}-\sigma_2}{2} + \frac{\sigma_{2,1}}{6}\right)^2 = 
\sigma_{1,1} + \sigma_2 + \frac{\sigma_{2,2}}{6}
\]
and 
\[(ab)^2 = \sigma_2^2 = \sigma_{2,2}.\] 
Now we can complete the computation we started in (\ref{ch_sym}). We have 
\begin{align*}
ch(Sym^t(\nu))=&(t+1) + {t+1 \choose 2} (a+b) + 
{t+1 \choose 3} \left((a+b)^2 - ab\right) + \\
&+{t+1 \choose 4} (a+b)\left((a+b)^2 - 2ab\right) + 
{t+1 \choose 5}\left(a^2b^2 + (a+b)^2\left((a+b)^2 - 3ab\right)\right).  
\end{align*}
Substituting our formulas for $ab$, $a+b$, $(ab)^2$, and $(a+b)^2$ we get
\begin{empheq}[box=\widefbox]{align}
\label{Sym_nu}
ch(Sym^t(\nu))= &(t+1) - {t+1 \choose 2}\sigma_1 
+\left(\frac{1}{2}{t+1 \choose 2} + {t+1\choose 3}\right)\sigma_{1,1} - \frac{1}{2}{t+1 \choose 2}\sigma_2 + \\
\nonumber&+\left(\frac{1}{6}{t+1 \choose 2} + \frac{1}{2}{t+1 \choose 3}\right)\sigma_{2,1} + \frac{1}{12}{t+1 \choose 3}\sigma_{2,2}		
\end{empheq}

\subsection{Chern Characters of $ch(Sym^t(\nu))$ - Second Computation.}
In this subsection, we shall compute the Chern character of $sym^t(\nu)$ using the Adams operations. The standard textbook about Adams operations is \cite[Chapter III]{A}. 

Recall some classical results from $K$-theory. 
Let $X$ be an algebraic variety and let $K^0(X)$ denote the $K$-theory of locally free sheaves of $\mathcal{O}_X$-modules, or equivalently vector bundles on $X$. 
For a vector bundle $E\to X$, we consider the power series 
\[
\Lambda_E(z)=\sum_{i=0}^\infty (-1)^{i} (\wedge^iE) z^i.
\]
Then $\Lambda_E(z)$ is multiplicative:
$\Lambda_{E\oplus E'}(z)=\Lambda_{E}(z)\Lambda_{E'}(z)$. 
Moreover, \[\Lambda_E(z)^{-1}=\sum_{i=0}^{\infty}Sym^i(E)z^i:=S_E(z).\]
Indeed, both power series are multiplicative and hence define two natural transformations of functors in $X$: 
$K^0(X)\to (K^0(X)[[z]])^\times$. To compare them, by the splitting principle it suffices to evaluate at line bundles. For a line bundle $L$ we have
\[\Lambda_L(z)S_L(z)=(1-Lz)\sum_{i=0}^\infty L^iz^i=1.\]

Then $d\log(\Lambda_E(z))$ is the generating function for the Adams operations, 
or more precisely \[\frac{\Lambda'_E(z)}{\Lambda_E(z)}=-\sum_{i=1}^\infty \psi^i(E)z^{i-1},\] where $\psi^i$ is the $i$-th Adams operation. Also, this follows from the additivity of $d\log(\Lambda_E(z))$ and the fact that on line bundles it gives $d\log(\Lambda_L(z))=\frac{L}{1-Lz}=-\sum_i z^{i-1}L^i=-\sum_i\psi^i(L)z^{i-1}.$
It follows that 
\[\frac{S_E'(z)}{S_E(z)}=-\frac{\Lambda'_E(z)}{\Lambda_E(z)}=\sum_{i=1}^\infty \psi^i(E)z^{i-1}.\]
To conclude, we get the classical formula  
\[S_E(z)=exp\left(\sum_{i=1}^\infty\frac{\psi^i(E)}{i}z^i\right).\]


Recall that we have a Chern character isomorphism \[ch:K^0_{top}(X)\otimes \mathbb{Q}\to H^{even}(X,\mathbb{Q}),\] and we can push the Adams operations along $ch$ to get maps $\psi_H^\ell : H^{even}(X,\QQ)\to H^{even}(X,\QQ)$. 

More precisely, we define $\psi^i_H: H^{even}(X,\QQ)
\to H^{even}(X,\QQ)$ by the commutativity of the diagram 

\[
\xymatrix{
	K^0_{top}(X)\otimes \QQ \ar^{ch}[r]\ar^{\psi^i}[d] & H^{even}(X,\QQ)\ar^{\psi^i_H}[d] \\
	K^0_{top}(X)\otimes \QQ \ar^{ch}[r] & H^{even}(X,\QQ) 
}
\]

The operation $\psi_H^\ell$ is given by the formula
\[\psi_H^\ell|_{H^{2i}(X)}=\ell^{i}.\]

The standard textbook about Adams operations is \cite{A}, chapter III.

As a consequence, we get the following formula for the Chern character of symmetric powers: 

\begin{theorem}
	Let $X$ be a topological space, and let $E$ be a vector bundle on $X$. Let 
	$ch(E)=\sum_is_i$ with $deg(s_i)=2i$. Then 
	\[ch(S_E(z))=exp\left(\sum_{i=1}^\infty \sum_k i^{k-1}z^is_k\right).\]
\end{theorem}
Set $\phi_k(z)=\sum_{i=1}^\infty i^kz^i$ we can rewrite this as 
\[ch(S_E(z))=exp\left(\sum_k \phi_{k-1}(z)s_k\right).\]
The functions $\phi_k(z)$ satisfies $\phi_{-1}(z)=-ln(1-z), \phi_0(z)=\frac{z}{1-z}$ and $\phi_{k+1}(z)=z\partial_z( \phi_k(z))$. Hence, we can compute them easily to any $k$ we want. More precisely, we claim that for
$f(z)=\frac{z}{1-z}$ we have $\phi_k(z)=\sum_{j=0}^k a^k_j (f(z))^j$ for some numbers $a_k^j$, which we shall now compute. 
We have $z\partial_z(f(z))=f(z)+f(z)^2$ and so 
\[z\partial_z(\sum_{j=0}^k a^k_j (f(z))^j) = \sum_{j=0}^k a^k_j j(f(z))^{j-1}(f(z)+ f(z)^2).\] 
We get the recursion
\[a^{k+1}_{j}=ja^k_{j}+(j-1)a^{k}_{j-1}.\] 
If we denote $p_k(z)= \sum_j a^k_jz^j$  then we see that $p_0(z)=z$ and in general 
\[p_k(z)=((z+z^2)\partial_z)^k z .\]
Having this sequence of polynomials then we can rewrite the recursion for $\phi_k$ as $\phi_k(z)=p_k(\frac{z}{1-z})$. 
For our purpose we need only the first few of them: 
\begin{itemize}
	\item $\phi_{-1}=-ln(1-z)$
	\item $\phi_0=f(z)$
	\item $\phi_1=f(z)+(f(z))^2$
	\item $\phi_2=f(z)+3f(z)^2+2f(z)^3$
\end{itemize}
Combining the formulas
\[ch(S_E(z)) = exp\left(\sum_k\phi_{k-1}(z)s_k(E)\right)\]
and
\[ch(\nu) = 2 -\sigma_1 +\frac{\sigma_{1,1}-\sigma_2}{2} + \frac{\sigma_{2,1}}{6},\]
we get 
\begin{align*}
ch(Sym(\nu)) &= exp\left(-2ln(1-z) -\sigma_1f +\frac{\sigma_{1,1}-\sigma_2}{2}(f+f^2) + 
\frac{\sigma_{2,1}}{6}(f+3f^2+2f^3))\right)=\\
=&\left(\frac{1}{1-z}\right)^2\left\{1-\sigma_1f +
\frac{\sigma_{1,1}}{2}(f+2f^2) - \frac{\sigma_{2}}{2}f + 
\frac{\sigma_{2,1}}{6}(f+3f^2)+
\frac{1}{12}\sigma_{2,2}f^2 \right\}.
\end{align*}

Using the fact that 
$f^a(1+f)^b=\sum_{m=a}^{\infty} {b-1 + m\choose a+b-1}z^m$ and $1+f=\frac{1}{1-z}$ we get the formula 

\begin{align*}
ch(Sym^t(\nu))&=(t+1) -\sigma_1{t+1 \choose 2} +
\frac{\sigma_{1,1}}{2}\left({t+1 \choose 2}+2{t+1 \choose 3}\right) - 
\frac{\sigma_{2}}{2}{t+1 \choose 2} + \\
&+\frac{\sigma_{2,1}}{6}\left({t+1 \choose 2}+3{t+1 \choose 3}\right)+
\frac{1}{12}\sigma_{2,2}{t+1 \choose 3}, 
\end{align*}
which agrees with the result of the previous computation (see equation (\ref{Sym_nu})).

\section{Computation of $N$}
We are interested in computing $\frac{1}{c(V_{k,d})}$ 
and then extracting the degree 4 term. 
 
We start by using the computation above to compute $ch(-[V_{k,d}]) = - ch(V_{k,d})$. 
Using the formula (\ref{K-class}) and the fact that $ch$ is a ring homomorphism,
we get 
\[ch(V_{k,d}) = {d-k+3\choose 3} ch(Sym^k(\nu)) - {d-k+2 \choose 3}ch(\wedge^2 \nu)ch(Sym^{k-1}\nu).\] 

We substitute (\ref{Sym_nu}) and (\ref{ch_wedge_2}):
\begin{align*}
&ch(\wedge^2 \nu)ch(Sym^{k-1}\nu)=
\left(1 -\sigma_1 + \frac{\sigma_{1,1}+ \sigma_2}{2} - 
\frac{\sigma_{2,1}}{3} + \frac{\sigma_{2,2}}{12}\right) \cdot \\ 
&\left(k - {k \choose 2}\sigma_1 +
\left(\frac{1}{2}{k \choose 2} +  {k\choose 3}\right)\sigma_{1,1} - 
{k \choose 2} \frac{\sigma_2}{2} +\left(\frac{1}{6}{k \choose 2} + 
\frac{1}{2}{k \choose 3}\right)\sigma_{2,1} + \frac{1}{12}{k \choose 3}\sigma_{2,2}\right)= \\
=& \,k - {k+1 \choose 2}\sigma_1 + 
\left(\frac{1}{2}{k+1\choose 2}+{k+1 \choose 3}\right) \sigma_{1,1} +\frac{1}{2}{k+1 \choose 2} \sigma_2 - \\
&-\left(\frac{1}{2}{k+1\choose 3} + \frac{1}{3}{k+1\choose 2}\right) \sigma_{2,1} 
+\frac{1}{12}{k+2\choose 3} \sigma_{2,2}.
\end{align*}

We denote $u=d-k$ for brevity. 
A standard computation, using (\ref{Sym_nu}) and the previous formula, of
\[ch(V_{k,d})={u+3 \choose 3} ch(Sym^k(\nu))-{u+2 \choose 3}ch(\wedge^2\nu)ch(Sym^{k-1}(\nu)),\]  
gives the coefficients of $ch(V_{k,d})$: 
\begin{itemize}
\item $ch_0(V_{k,d})={u + 3 \choose 3} + k {u+2 \choose 2}$
\item $ch_1(V_{k,d})=-{u+2 \choose 2} {k+1 \choose 2} \sigma_1$
\item $ch_2(V_{k,d})={u+2 \choose 2 } \left(\frac{1}{2}{k+1 \choose 2} + {k+1 \choose 3}\right) \sigma_{1,1} - 
\frac{1}{2}\left({u+3 \choose 3} + {u + 2\choose 3}\right){k+1 \choose 2}\sigma_2$
\item $ch_3(V_{k,d})=\frac{2ku+3k+u}{18}{u+2 \choose 2}{k+1 \choose 2} \sigma_{2,1}$ 
\item $ch_4(V_{k,d})=\frac{1}{12}\left({u+3 \choose 3}{k+1 \choose 3}  - 
{u+2 \choose 3}{k+2\choose 3}\right)\sigma_{2,2}$
\end{itemize}	 
 
We use the Newton identities to translate from the Chern character to the Chern classes. Assume that $s_l=\sum_i x_i^l$ and $c_l$ are the elementary symmetric polynomials in the variables $x_1,...,x_n,...$. Then 
\[c_4 = \frac{s_1^4 + 8 s_3 s_1 - 6 s_1^2 s_2 + 3 s_2^2 -6s_4}{24}.\] 
 
We want to compute the degree 4 component of $c(-[V_{k,d}])$. Writing  
$c(-[V_{k,d}])=\prod_{i=1}^\infty(1+\alpha_i)$ we have a formula for 
$\sum_{i=1}^{\infty} (e^{\alpha_i}-1) = ch(dim(V_{k,d})[\CC]-[V_{k,d}])$ which we know.  

The relation between $s_l$ and the coefficients of $ch(V_{k,d})$ is that 
$s_l(-[V_{k,d}]) = -(ch_l(V_{k,d}) \cdot l!)$. 
Using this formula and the computation of the Chern character of $V_{k,d}$ 

\begin{itemize}
	\item $s_1 = {u+2 \choose 2} {k+1 \choose 2} \sigma_1$
	\item $s_2 = -{u+2 \choose 2 }\left({k+1 \choose 2} + 2{k+1 \choose 3}\right) \sigma_{1,1}+
	\left({u+3 \choose 3} + {u + 2\choose 3}\right){k+1 \choose 2}\sigma_2$
	\item $s_3 = -\frac{2ku+3k+u}{3}{u+2 \choose 2}{k+1 \choose 2}$
	\item $s_4 = 2\left(-{u+3 \choose 3}{k+1 \choose 3} + {u+2 \choose 3}{k+2\choose 3}\right)\sigma_{2,2}$        
\end{itemize} 

Or equivalently
\begin{itemize}
	\item $s_1 = \frac{1}{4}(u+2)(u+1)(k+1)k \sigma_1$
	\item $s_2 = \frac{1}{12}(u+2)(u+1)(k+1)k((2u+3) \sigma_{2} - (2k+1)\sigma_{1,1})$
	\item $s_3 = -\frac{1}{12}(u+2)(u+1)(k+1)k(2ku+3k+u)\sigma_{2,1}$
	\item $s_4 = \frac{2}{6^2}(u+2)(u+1)(k+1)k\left(-(u+3)(k-1) + u(k+2)\right)\sigma_{2,2}$
\end{itemize}

If we now set $\phi = \frac{1}{12}(u+2)(u+1)(k+1)k$ we can rewrite this as 

\begin{itemize}
	\item $s_1 = 3\phi \sigma_1$
	\item $s_2 = \phi((2u+3) \sigma_{2} - (2k+1)\sigma_{1,1})$
	\item $s_3 = -\phi(2ku+3k+u)\sigma_{2,1}$
	\item $s_4 = 2\phi(u-k+1)\sigma_{2,2}$ 
\end{itemize}

We deduce that 
\begin{align*}
24 c_4 =& s_1^4 + 8 s_3 s_1 - 6 s_1^2 s_2 + 3 s_2^2 -6s_4 = \\
=& (3\phi \sigma_1)^4 + 8 (-\phi(2ku+3k+u)\sigma_{2,1}) 3\phi \sigma_1 - 
6 (3\phi \sigma_1)^2 (\phi((2u+3) \sigma_{2} - (2k+1)\sigma_{1,1})) +\\
&+ 3 (\phi((2u+3) \sigma_{2} - (2k+1)\sigma_{1,1}))^2 
-12\phi(u-k+1)\sigma_{2,2}=\\ 
=3\phi& \left[ 2 \cdot 3^3 \phi^3 -4(9 \phi^2+1)(u-k+1) + 2\phi \left(2d^2-12ku+2u-10k+5\right)  
\right]\sigma_{2,2}.   
\end{align*}
Therefore the number of surfaces is

\[\boxed{N=\frac{\phi}{4} \left[  (3 \phi)^3 -2(9 \phi^2+1)(u-k+1) + 
\phi \left(2d(d+1)-12k(u+1)+5\right)  \right]}\]

\section{Remarks on Generic Position}
We have computed the intersection index of some homology classes, but does it count the actual number
of surfaces containing given $\delta$ points? For some particular choice of $\delta$ points, 
it might not be the case; but for $\delta$ generic points, we claim that $N$ is precisely the answer to the enumerative question.

There are two issues to be addressed here. The first issue is to prove that the intersection of the cycles, for a generic choice of $\delta$ points, is transversal. 

The second issue is whether a generic surface of the kind we describe has just one line of order $k$. It holds for almost every $d$ and $k$ with one particular exception.

Here we shall discuss both issues.

\subsection{Transversality}

In the computation of $N$, we intersect $\delta$ cycles in $\PP V_{k,d}$.
For each point $P\in \mathbb{P}^3$ we have the cycle $H_P=\{(\ell,S)\in \PP V_{k,d} : P\in S\}$ 
and we choose $\delta$ such cycles to compute the intersection. 
Hence, we want to show that for generic tuple $P_1,...,P_\delta$, 
the cycles $H_{P_1},...,H_{P_\delta}$ intersect transversely in 
$\PP V_{k,d}$. Since transverality is an open condition, it suffices 
to find one such tuple with a transversal intersection point. 
We may start with a chosen surface $S$ 
and choose the points $P_1,...,P_\delta$ accordingly. 

Consider the family of surfaces $S: x^kz^{\alpha}t^{d-k-\alpha}-y^kz^{\beta}t^{d-k-\beta}=0$ for $0\le \beta <\alpha \le d-k$. We shall 
introduce local coordinates on $\PP V_{k,d}$ near $S$. 
To variate the line, we replace the line $x=y=0$ 
by a nearby line $u=v=0$, where $u=x+Az+Bt$ and $v=y+Cz+Dt$.
For $A,B,C,D\in \mathbb{C}$ and $a_{p,q,r,s}\in \mathbb{C}$ for $p+q+r+s=d$, 
$p+q\ge k$, except $(p,q,r,s)=(0,k,0,d-k)$, consider the surface 
\[ S_{A,B,C,D,\{a_{p,q,r,s}\}}: \left\{
u^kz^{\alpha}t^{d-k-\alpha}-v^kz^{\beta}t^{d-k-\beta}+
\sum_{p,q,r,s}a_{p,q,r,s}u^pv^qz^rt^s = 0 \right\}.\]

The association $(A,B,C,D,\{a_{p,q,r,s}\})\mapsto (\left\{u=v=0\right\}, S_{A,B,C,D,\{a_{p,q,r,s}\})}$ gives 
an open embedding $\mathbb{C}^\delta \to \PP V_{k,d}$. Hence, 
we can use the coordinates $A,B,C,D,\{a_{p,q,r,s}\}$ to compute the Jacobian matrix 
of the set of equations $H_{P_1},...,H_{P_\delta}$ and verify transversality. 

For the points $P_1,P_2,P_3,P_4$ set 
\[P_i=(\epsilon ,\epsilon  i,i^{\frac{k}{\alpha-\beta}},1),\]
for $\epsilon$ a parameter which will be taken small enough in the end. 
The points $P_5,...,P_\delta$ will be chosen generically.

It is easy to see that all these points are on $S$.

Now we shall compute the partial derivatives the defining equations of the $H_{P_i}$ 
with respect to $\frac{\partial}{\partial A}$, $\frac{\partial}{\partial B}$,
$\frac{\partial}{\partial C}$, $\frac{\partial}{\partial D}$,
and then by $\frac{\partial}{\partial a_{p,q,r,s}}$, where $a_{p,q,r,s}$ 
is the coefficient of $u^pv^qz^r$, for all monomials with $p+q\geq k$
except $y^kt^{d-k}$. 

We have, for $P=(x,y,z,1)$:
\begin{align*}
&\frac{\partial H_P}{\partial A} = kx^{k-1}z^{\alpha+1} \\ 
&\frac{\partial H_P}{\partial B} = k^{k-1}z^{\alpha} \\
&\frac{\partial H_P}{\partial C} = -kv^{k-1}z^{\beta+1} \\
&\frac{\partial H_P}{\partial D} = -kv^{k-1}z^{\beta} \\
&\frac{\partial H_P}{\partial a_{p,q,r,s}} = x^p y^q z^r t^s 
\end{align*}

and we want to show that the determinant of the matrix

\[J=
\begin{pmatrix}
\frac{\partial H_{P_{1,2,3,4}}}{\partial \{A, B, C, D \} } & 
\frac{\partial H_{P_{1,2,3,4}}}{\partial \{ a_{p,q,r,s} \} } \\
\\
\frac{\partial H_{P_{5,...,\delta}}}{\partial \{A, B, C, D \} } & 
\frac{\partial H_{P_{5,...,\delta}}}{\partial \{ a_{p,q,r,s} \} }
\end{pmatrix}
\]
is non-zero for some $\epsilon$. 
Here we separate the matrix into 4 blocks: the first 4 rows are separated
from the last $\delta-4$, and same with columns.

The upper-left block is
\[
\frac{\partial H_{P_{1,2,3,4}}}{\partial \{A, B, C, D \} }
=   
\begin{pmatrix}
\frac{\partial H_{P_1}}{\partial A} & \hdots & \frac{\partial H_{P_1}}{\partial D}\\
\vdots	& \ddots	& \vdots\\
\frac{\partial H_{P_4}}{\partial A} & \hdots & \frac{\partial H_{P_4}}{\partial D}
\end{pmatrix}=k\epsilon^{k-1}
\begin{pmatrix}
1 & 1 & -1 & -1\\
2^{k\frac{\alpha+1}{\alpha-\beta}} & 2^{k\frac{\alpha}{\alpha-\beta}} & -2^{k-1+\frac{\beta+1}{\alpha-\beta}} & -2^{k-1 + \frac{\beta}{\alpha-\beta}k} \\
3^{k\frac{\alpha+1}{\alpha-\beta}} & 3^{k\frac{\alpha}{\alpha-\beta}} & -3^{k-1+k\frac{\beta+1}{\alpha-\beta}} & -3^{k-1 + k\frac{\beta}{\alpha-\beta}} \\
4^{k\frac{\alpha+1}{\alpha-\beta}} & 4^{k\frac{\alpha}{\alpha-\beta}} & -4^{k-1+k\frac{\beta+1}{\alpha-\beta}} & -4^{k-1 + k\frac{\beta}{\alpha-\beta}} \\
\end{pmatrix}.
\]

We shall need a simple lemma from classical analysis cf. \cite{PS}, book II, part 5, $\S 4$, problem 48.

\begin{lemma}
	For two finite sets of distinct positive numbers $0<a_1<...<a_n$ and $0<b_1<...<b_n$, the matrix $M$ with entries $m_{i,j}={a_i}^{b_j}$ satisfies $det(M) \neq 0$.
\end{lemma}








\begin{proof}
	It is a slightly more general version of Vandermonde determinant.
	Would it not be so, there would be a function $p(x)=c_1x^{a_1}+c_2x^{a_2}+...+c_nx^{a_n}$ 
	(like a polynomial, but with non-integer powers), 
	which has at least as many positive roots as monomials. 
	One might divide the function by $x^{a_1}$, and now it has a constant as one of its monomials, 
	but the same roots, and then consider its derivative: it has at least $n-1$ roots by Rolle's theorem and just $n-1$ monomials. 
	Proceeding in this way, we end up with a single monomial and at least one positive root, which is impossible.
\end{proof}

It follows that, whenever the numbers 
$k\frac{\alpha+1}{\alpha-\beta}, k\frac{\alpha}{\alpha-\beta}, k-1+k\frac{\beta+1}{\alpha-\beta}, k-1 + k\frac{\beta}{\alpha-\beta}$ are pairwise distinct, the matrix 
$\frac{1}{\epsilon^{k-1}}\frac{\partial H_{P_{1,2,3,4}}}{\partial \{A, B, C, D\}}$ is non-degenerate. 
It is always possible to choose such $\alpha$ and $\beta$.

If $\alpha=1$ and $\beta=0$, then the exponents are $2k, k, 2k-1, k-1$.
Those numbers are distinct except when $k=1$.

If $k=1$, we take $\alpha=2$ and $\beta=0$ the exponents are
$\frac{3}{2}, 1, \frac{1}{2}, 0$.
If $d-k=1$ we are not allowed to choose $\alpha=2$. 

But anyway we disregard the case $d=2, k=1$, 
since there are infinitely many lines on a quadric surface.


On the other hand, 
it is easy to see that the upper right block 
$\frac{\partial H_{P_{1,2,3,4}}}{\partial \{ a_{p,q,r,s} \} }$ 
of $J$ divisible by $\epsilon^k$. 

We can deform the surface replacing the equation 
$x^kz^{\alpha}t^{d-k-\alpha}-y^kz^{\beta}t^{d-k-\beta}=0$ 
with the more general equation 
\[S_{A,B,C,D,\{a_{p,q,r,s}\}}:
u^kz^{\alpha}t^{d-k-\alpha}-v^kz^{\beta}t^{d-k-\beta}+
\sum_{p,q,r,s}a_{p,q,r,s}u^pv^qz^rt^s.\] 
Then we want to deform the curves $P_1(\epsilon),...,P_4(\epsilon)$
so that they would still belong to surface and such that $P_i(0)$ would 
be on the singular line.
We take $P_i(\epsilon)$ to be
$(u_i=\epsilon ,v_i=\epsilon  i,z_i,t=1 )$, where $z_i$ will be defined 
as follows. We write the equation of $S$, substitute the values of $u_i, v_i$ and $t=1$,
so we obtain a polynomial equation in $z$ of degree $d-k$, with coefficients 
divisible by $\epsilon^k$; if we divide by $\epsilon^k$ the leading coefficient 
will be non-zero for a sufficiently small deformation. 
We choose a root $z$ of this equation analytically, 
so that in the original position, we get $z_i=i^{\frac{k}{\alpha-\beta}}$. 
Then, after deformation, we get $z_i\approx i^{\frac{k}{\alpha-\beta}}$.

It follows that for surfaces in a small neighborhood of $S$, 
one can find one-parameter families $P_1(\epsilon),...,P_4(\epsilon)$ on the surface,
such that $\frac{1}{\epsilon^{k-1}}\frac{\partial H_{P_{1,2,3,4}}}{\partial \{A, B, C, D\}}$ is non-degenerate. 
Notice that the condition that the block 
$\frac{\partial H_{P_{1,2,3,4}}}{\partial \{ a_{p,q,r,s} \} }$
is divisible by $\epsilon^k$ remains true in the deformation, since $u_i$ and $v_i$ are divisible by $\epsilon$.

Hence it suffices to show that for Zariski-generic choice of a surface and $\delta-4$
points, the lower-right block $\frac{\partial H_{P_{5,...,\delta}}}{\partial \{ a_{p,q,r,s} \} }$
in invertible. 


Geometrically, the non-degeneracy of this block correspond to the transversality of the intersections $H_{P_i}\cup \PP(V_{k,d}|_{\ell})$ at $S$ for $i=5,...,\delta$. 

The intersection $H_{P_i}\cap \PP(V_{k,d}|_\ell)$ is the linear subspace of surfaces containing the line $k$ times and containing the point $P_i$. It suffices to show, then, that the functionals defined by the substitution of the points $P_i$ in homogeneous polynomials defining such $S$-s are linearly independent. 
This is guaranteed if we choose the point $P_{i+1}$ not to lie on some surface containing $\ell$ to order $k$ and the points $P_1,...,P_i$.  

We deduce that the generic intersection of cycles of the form $H_{p_i}$ is transversal.

\subsection{Uniqueness of a Line}
We would like to show that, for $(d,k)\ne (3,1)$, the generic surface containing a line $k$-times contains a unique such line. To show this, it would suffice to prove that the dimension of the space of surfaces containing two lines $k$-times is strictly smaller than the dimension of $SingL_{k,d}$. 

\begin{theorem}
	Let $DoubL_{k,d}$ denote the subvariety of $\PP(Sym^d((\mathbb{C}^4)^*))$ of all those surfaces containing 2 lines $k$-times. If $d\ge 4$ or $k\ge 2$ then 
	\[dim(DoubL_{k,d})<dim(SingL_{k,d}).\]
\end{theorem}

\begin{proof}
	We shall count degrees of freedom, hence bounding $dim(DoubL_{k,d})$ from above. A pair of distinct lines in $\PP^3$ may be either disjoint or intersecting at a point. 
	
	\textbf{Intersecting lines.}
	If the two lines intersect in a point, then up to projective transformation we may assume that $\ell_1=\{x=y=0\}$ and $\ell_2=\{y=z=0\}$. The dimension of the space of such pairs is $2dim(\Gr)-1=7$. Hence, it suffices to prove that the condition of vanishing on $\ell_2$ cuts a subspace of codimension greater than $3$ in $\PP(V_{k,d}|_{\ell_1})$. For a polynomial of the form 
	$\sum_{p+q\ge k, p+q+r+s=d}a_{p,q,r,s}x^py^qz^rt^s$, the condition that it vanishes to order $k$ on $\ell_2$ is equivalent to the conditions 
	$a_{p,q,r,s}=0$ if $q+r<k$. So it suffices to prove that there are at least $4$ monomials $x^py^qz^rt^s$ with $p+q\ge k, p+q+r+s=d, q+r<k$. Now consider the following example: 
	\[\begin{cases}
	(x^d,x^{d-1}t,x^{d-2}t^2,x^{d-3}t^3) & k=1, d > 3 \\
	(x^d,x^{d-1}y,x^{d-1}t,x^{d-1}z)     & 1<k<d \\
	\end{cases}\]
	
	For the case $k=d$, just note that the equation $f$ of a  surface $S$ of degree $d$ passing $d$ times through intersecting lines $\ell_1$ and $\ell_2$ must be of the form $f^d=0$ where $f$ is a linear function. Hence, the dimension of such surfaces is at most $3$, which is smaller than the dimension of $SingL_{d,d}$ for every $d>1$. 
	
	In the case $d=3,k=1$, we know that the line is not unique. In fact, every surface contains 27 lines. In the case, $d=2,k=1$, each surface contains infinitely many lines, and we shall not consider this case. 
	
	\textbf{Disjoint lines.}
	Up to projective transformation, every pair of disjoint lines can be transformed to $\ell_1=\{x=y=0\}$ and $\ell_2=\{z=t=0\}$. The dimension of pairs of disjoint lines is $8$ so it suffice to prove that vanishing on $\ell_2$ amounts to at least $5$ linearly independent conditions. Consider the examples  
	\[\begin{cases}
	(x^d,x^{d-1}y,...,y^d) & d\ge 4 \\
	(x^d,...,y^d,x^{d-1}z,x^{d-1}t)     & 1<d<4,k>1 \\
	\end{cases}\]
\end{proof}

The case $d=3, k=1$ is special and will be discussed in the next section.
\section{Particular Cases}
\subsection{A Problem in Combinatorial Geometry.}
Consider a particular case: $d=k$. It corresponds to a union of $d$ planes, having a common line.
In this case, we need $d+4$ generic points. 

We may compute the number of such plane arrangements as a particular case of our formula,
or by elementary combinatorial argument.

Let us start with our formula. If $u=0$, $k=d$ the formula simplifies to
\[24 N =6\phi \left[  (3 \phi)^3+2(9 \phi^2+1)(k-1) + \phi \left(2k^2 -10k+5\right)\right]=\]
\[=3\phi \left[ 2 \cdot 3^3 \phi^3 +2^23^2 \phi^2 (k - 1) + \phi \left(4k^2 -20k+10\right)+4(k-1)\right].\]

In our case $\phi = \frac{k(k+1)}{6}$, hence
\begin{align*}
24 N &= 
2\left(\frac{k(k+1)}{2}\right)^4  + 4\left(\frac{k(k+1)}{2}\right)^3(k-1)+\\
& +\frac{(k(k+1))^2}{12}\left(4k^2 -20k+10\right)+2k(k+1)(k-1),
\end{align*}

which simplifies to  
\[N = \frac{(k-1)k(k+1)(k+2)(k+3)(k+4)}{24^2}\left(3k^2-3k+2\right).\]

The same question can be solved by an elementary combinatorial argument.
Since we have $k+4$ points and only $k$ planes, some planes must contain 
more than one point. No plane is allowed to contain 4 points, since points
are generic. So there are $c$ planes with 3 points and $b$ planes with 2 points,
where $2c+b=4$. This gives 3 possibilities:

\begin{itemize}
	\item $c=2$, $b=0$,
	\item $c=1$, $b=2$,
	\item $c=0$, $b=4$.
\end{itemize}

In the first case, we choose 2 triples of points, each defining a plane.
The intersection of the two planes is a uniquely determined line.
Each extra point is contained in a unique plane containing that line.
So the number of ways to choose such a configuration of planes equals the number of ways to select two triples of points among $k+4$ given points, 
which gives $\frac{(k+4)!}{(k-2)!3!3!2}$ possibilities.

In the second case, 3 points specify a plane. The other two additional pairs specify lines, which have to intersect the common line of all planes.
These two lines generically intersect the first plane in 2 distinct points. 
In the plane, a line connecting 2 points is unique. Then we must select additional planes, connecting that unique line with all other points.
So configurations of this type are determined by choices of a triple and two pairs among the points. 
This gives $\frac{(k+4)!}{(k-3)!3!2^3}$ configuration.

In the third case, we have 4 pairs of points, each determining a line in space. The common line of all planes must intersect those four given lines. It is well known that for generic 4 lines,
there are precisely 2 lines intersecting all of them. So this time, 
the number of combinatorial choices must be multiplied by 2.
We have to choose 4 pairs of points, which gives 
$\frac{(k+4)!}{(k-4)!4!2^4}$ combinatorial choices, and this number has to be multiplied by 2 for the geometric reason: there are two choices for the lines intersecting the four chosen lines.

Summing up those numbers, we get 
\begin{align*}
N&=\frac{(k+4)!}{(k-2)!3!3!2}+\frac{(k+4)!}{(k-3)!3!2^3}+\frac{(k+4)!}{(k-4)!4!2^3}\\
&=\frac{(k+4)(k+3)(k+2)(k+1)k(k-1)}{24^2}\left(3k^2-3k+2\right).
\end{align*}

\subsection{Twenty Seven Lines on a Cubic Surface.}
Another particular case is the famous theorem of Cayley and Salmon about 27
lines on a generic cubic surface.

Cubic surfaces are parametrized by $\PP^{19}$. By linear algebra, a set of $19$ generic points determines
a unique cubic surface. If a generic cubic surface has $N$ lines, this cubic surface would be counted $N$ times. We might substitute $d=3$ and $k=1$ to our formula and find $N$.

In our notations, $u=d-k=2$ and  $\phi = \frac{1}{12}(u+2)(u+1)(k+1)k=2$, and hence

\[N=\frac{2}{4} \left[  (3 \cdot 2)^3 -2(9 \cdot 2^2+1)(2-1+1) + 2 \left(2\cdot 3\cdot 4-12\cdot 3+5\right)
\right]=27.\]

\subsection{Cubic scrolls.}
The remaining particular case in degree 3 is cubic scrolls through $13$ generic points.
A general cubic scroll in $\PP^3$ has a double line. 
Conversely, a general cubic surface with a double line is a cubic scroll. 
Substituting $d=3$ and $k=2$ , we get $u=1$ and  $\phi = 3$, and hence
\[N=\frac{3}{4} \left[  9^3 -2(9\cdot 3^2+1)(1-2+1) + 3 \left(2\cdot 3\cdot 4-12\cdot 4+5\right)  \right] = 504.\]

This number has been previously computed by several methods (see \cite{CV,C,M}). 

\subsection{Table of Small Values} We list several values given by the formula. 
For $d=3$ and $k=1$, the number of surfaces should be divided by 27.
\medskip
\medskip
\begin{center}
	\begin{tabular}{c|c|c|c|c|c|c}
		\backslashbox{k}{d}& 2&3 &4 &5 &6 &7  \\\hline
		1& 0 	& 27	& 320 	& 1990	& 8680		& 29960		\\\hline
		2& 10 & 504	& 7530	& 57715	& 294390	& 1143030	\\\hline
		3& 		& 175	&	9480	&	138420&	1038490 &	5210295 \\\hline
		4& 		& 		& 1330	& 75735	& 1110510	& 8312500	\\\hline
		5& 		& 		& 	 		& 6510	& 385560	& 5696355	\\\hline
		6& 		& 		& 			&				& 24150		& 1476510	\\\hline
		7& 		& 		& 			&				& 				& 73920		\\
	\end{tabular}
\end{center}


\medskip
\begin{thebibliography}{99}
\bibitem[A]{A} M. Atiyah. {\it K-theory.} CRC Press, (2018).

\bibitem[C]{C} I. Coskun. {\it Gromov-Witten invariants of jumping curves}, Transactions of the American Mathematical Society,  360, p. 989--1004 (2008).

\bibitem[CLV]{CLV} F. Cukierman, A. F. Lopez, I. Vainsencher. {\it Enumeration of surfaces containing an elliptic quartic curve. } Proceedings of the American Mathematical Society 142.10, p.3305--3313 (2014).

\bibitem[CV]{CV} D. F. Coray and I. Vainsencher. {\it Enumerative formulae for ruled cubic surfaces and rational quintic curves}, Commentarii Mathematici Helvetici, 61, p.510--518  (1986).

\bibitem [GKZ] {GKZ} I. M. Gelfand, M. Kapranov, and A. Zelevinsky. {\it Discriminants, resultants, and multidimensional determinants.} Springer Science \& Business Media, (2008).

\bibitem[GH]{GH} P. Griffiths, J. Harris. {\it Principles of Algebraic geometry.}
 John Wiley \& Sons, (2014).

\bibitem [H]{H} R. Hartshorne. {\it Algebraic geometry}. Vol. 52. Springer Science \& Business Media, (2013).

\bibitem[KKN]{KKN}  M. Kazarian, D. Kerner, and A. Némethi. {\it Discriminant of the ordinary transversal singularity type. The global equivalence class.} arXiv preprint arXiv:1308.6045, (2013).

\bibitem[K]{K} D. Kerner. {\it Enumeration of uni-singular algebraic hypersurfaces.} 
Proc. Lond. Math. Soc. (3) 96, no. 3, p.623--668 (2008).

\bibitem[L]{L} R.K. Lazarsfeld. {\it Positivity in algebraic geometry I: Classical setting: line bundles and linear series.} Vol. 48. Springer, 2017.

\bibitem[M]{M} C. Martínez . {\it The degree of the variety of rational ruled surfaces and Gromov-Witten invariants.} Transactions of the American Mathematical Society, 358(6), p.2611--2624  (2006).

\bibitem[PS]{PS} G. Pólya, G. Szegö. {\it Problems and Theorems in Analysis II.} Springer Science \& Business Media, (1997).




\end{thebibliography}
\end{document}